\documentclass{amsart}

\usepackage[T1, T2A]{fontenc}
\usepackage[russian, english]{babel}
\usepackage{amsmath}
\usepackage{amsfonts}
\usepackage{amsthm}
\usepackage{amssymb}
\usepackage{dsfont} %for \mathds{1}
\usepackage{todonotes}
\usepackage{color}

\newcommand{\supp}{\mathop{\rm supp}\nolimits}
\newcommand{\Diam}{\mathop{\rm Diam}\nolimits}

\newcommand{\Leb}{\mathop{\rm Leb}\nolimits}
\newcommand{\A}{\mathbb{A}}
\newcommand{\R}{\mathbb{R}}
\newcommand{\N}{\mathbb{N}}
\newcommand{\Z}{\mathbb{Z}}

\newcommand{\p}{\mathbb{P}}

\newcommand{\mM}{\mathcal{M}}

\newcommand{\dd}{\, \mbox{d}}
\newcommand{\mgr}{\mu}
\newcommand{\msp}{\nu}

\newcommand{\eps}{\varepsilon}
\newcommand{\1}{\mathds{1}}

\newcommand{\emr}[1]{\textcolor{black}{#1}}

\newtheorem{thm}{Theorem}[section]
\newtheorem{lm}[thm]{Lemma}
\newtheorem{prop}[thm]{Proposition}

\theoremstyle{definition}
\newtheorem{defn}[thm]{Definition}

\theoremstyle{remark}
\newtheorem{rem}[thm]{Remark}

\numberwithin{equation}{section}

\begin{document}

\title[Ergodic theorem for nonstationary random walks]{Ergodic Theorem for nonstationary random walks on compact abelian groups}

\author{Grigorii Monakov}

\address{Grigorii Monakov, Department of Mathematics, University of California, Irvine, CA~92697, USA}

\email{gmonakov@uci.edu}


\date{}

%\dedicatory{}

%    "Communicated by" -- provide editor's name; required.
\commby{}

\begin{abstract}
    We consider a nonstationary random walk on a compact metrizable abelian group. Under a classical strict aperiodicity assumption we establish a weak-* convergence to the Haar measure, Ergodic Theorem and Large Deviation Type Estimate.
\end{abstract}

\maketitle

\section{Introduction}

\subsection{Main result}

Consider a topological compact metrizable abelian group \\ $(\A, 0, +)$ with shift-invariant metric\footnote{It is worth noticing that the statement that all compact separable abelian groups of real power are metrizable can be neither proved nor disproved in ZFC \cite{H}. However, it is well-known (see \cite{Bi} and \cite{K}) that any topological group is metrizable if and only if it is first countable, and, moreover, the metric can always be taken to be left-invariant.} $d$:
\begin{equation*}
    d(x + \alpha, y + \alpha) = d(x, y) \quad \forall \alpha, x, y \in \A.
\end{equation*}
Let $\mathcal{B}$ denote the Borel $\sigma$-algebra. It is well-known \cite[Chapter XI]{Ha} that there exists a unique shift invariant Borel probability measure on $\A$ that is usually referred to as Haar measure. Let us denote this measure by $h$.

Denote the space of all probability measures on $(\A, \mathcal{B})$ by $\mM$. We define convolution of two measures $\mgr, \msp \in \mM$ by the following formula:
\begin{equation} \label{convAbelian}
    (\mgr * \msp) (B) = \int_{\A} \1_{B} (\alpha + \beta) \dd\mgr(\beta) \dd\msp(\alpha) \quad \forall B \in \mathcal{B}.
\end{equation}
One can view the measure $\mgr * \msp$ as a distribution of the image of a random point $\alpha \in \A$, distributed with respect to $\msp$, shifted by an independent random element $\beta \in \A$, distributed with respect to $\mgr$.

We will need the following \textbf{strict aperiodicity} condition (see, for instance, \cite{Bo1}):
\begin{defn} \label{def:transit}
    We say that a random shift defined by a measure $\mgr$ is \textbf{strictly aperiodic} if the smallest semigroup containing $\{ \alpha - \beta \; | \; \alpha, \beta \in \supp(\mgr)\}$ is dense in $\A$ (we denote by $\supp(\mgr)$ the topological support of the measure $\mgr$, i. e. the smallest closed subset of a full measure).
\end{defn}
It is not hard to see that the condition above is equivalent to $\supp(\mgr)$ not being contained in a coset of any proper closed (normal) subgroup of $\A$.

The main result of this paper can be stated as follows:

\begin{thm} \label{thm:mainGlobal}
    Consider a compact (in weak-* topology) $K \subset \mM$, such that any $\mgr \in K$ is \textbf{strictly aperiodic}. Let $\mgr_1, \mgr_2, \mgr_3, \ldots$ be a sequence of measures in $K$.

    Then for any $\varphi\in C(\A, \mathbb{R})$ and any $x \in \A$ almost surely
    \begin{equation} \label{isomErg}
        \frac{1}{n} \sum_{k = 0}^{n - 1} \varphi(x + \alpha_1 + \alpha_2 + \ldots + \alpha_k) \to \int_{\A} \varphi \dd h \ \ \text{\rm as}\ \ \ n\to \infty,
    \end{equation}
    where the elements $\alpha_1, \alpha_2, \ldots$ are independent random elements in $\A$ and $\alpha_j$ is distributed with respect to $\mgr_j$.
    %where $\nu_1=\delta_x$, and $\nu_{n+1}=\mu_n*\nu_n$ for all $n\ge 1$.
    %where the measures $\msp_n$ are defined by~\eqref{eq:nu-def}.

    Moreover, an analogue of the Large Deviations Theorem holds. Namely, for any $\eps > 0$ there exist $C, \delta > 0$ such that for any $x \in \A$ and any $n \ge 1$
    \begin{equation} \label{isomLD}
        %\forall n \in \mathbb{N}, \quad
        \p \left( \left| \frac{1}{n} \sum_{k = 0}^{n - 1} \varphi(x + \alpha_1 + \alpha_2 + \ldots + \alpha_k) - \int_{\A} \varphi \dd h \right| > \eps \right) < C \exp(-\delta n),
    \end{equation}
    where
    \begin{equation*}
        \p = \prod_{n=1}^\infty \mgr_n.
    \end{equation*}
\end{thm}

Some classical examples of abelian groups to which our results can be applied include flat tori, $p$-adic integers, and Cantor dyadic group. Additional motivation for studying stationary and nonstationary random shifts on tori comes from the recent interest to dynamics of mixed random-quasiperiodic cocycles, see \cite{CDK1, CDK2}.

\begin{rem}
    The proof presented in this paper relies heavily on the commutativity of $\A$. Nevertheless, we expect a similar result to hold in a noncommutative case as well and plan to study it in future publications.
\end{rem}

\subsection{Random and nonstationary ergodic theorems}

Ergodic Theorems and Large Deviation Type Estimates have been intensively studied for both stationary and nonstationary stochastic dynamical systems. Let us mention some classical and recent results in that area.

Let $X$ be a compact metric space and $\mgr$ a Borel probability measure on the space of continuous functions $C(X, X)$. For any Borel probability measure $\msp$ on $X$ we can define its convolution with $\mgr$ by
\begin{equation} \label{convGeneral}
    (\mgr * \msp) (A) = \int_{C(X, X)} \msp(f^{-1}(A)) \dd\mgr(f)
\end{equation}
for any Borel set $A \subset X$. The resulting measure $\mgr * \msp$ can be understood in the following way: it is the distribution of $f(x)$, where $x \in X$ is distributed with respect to measure $\msp$, the map $f \in C(X, X)$ is distributed with respect to the measure $\mgr$ and $f$ and $x$ are chosen independently from one another. Notice that formula (\ref{convAbelian}) is a special case of (\ref{convGeneral}).

A probability measure $\msp$ on $X$ is called stationary if $\mgr * \msp = \msp$. A stationary measure always exists. For instance, one can observe that any limit point of the sequence
\begin{equation*}
    \frac{1}{n} \sum_{j = 0}^{n - 1} \mgr^{*j} * \msp
\end{equation*}
will be a stationary measure (for details see \cite[Lemma 2.10]{BQ}). Now the Birkhoff Ergodic Theorem for stationary stochastic dynamical systems can be formulated as follows \cite[Theorem 5.2.9]{DMPS}:
\begin{thm}
    Let $\msp$ be any $\mgr$-stationary ergodic (extremal point of the set of $\mgr$-stationary measures) measure on $X$ and $\varphi \in C(X, \R)$ be a continuous function. Then for $\msp$-almost every point $x \in X$ and $\mgr^{\N}$-almost every sequence $f_1, f_2, \ldots \in C(X, X)$ one has
    \begin{equation*}
        \frac{1}{n} \sum_{k = 0}^{n - 1} \varphi(f_{k} \circ f_{k - 1} \circ \ldots \circ f_1 (x)) \to_{n \to \infty} \int_{X} \varphi(x) \dd\msp(x).
    \end{equation*}
\end{thm}

Large Deviation Type Estimates are also well studied for the stationary case. For that setting it turns out to be useful to operate in terms of Markov processes instead of stochastic dynamical systems. Some of the results can be found in \cite{V, A, DK} and references therein. In their recent paper Cai, Duarte and Klein obtained an abstract Large Deviations Theorem (see Theorem 1.1 and Remark 2.3 in \cite{CDK3}). One application of this result is a Large Deviation Theorem for random translations on the torus (see Theorem 1.2). Although this result is similar to our Theorem \ref{thm:mainGlobal}, it provides a more accurate estimate for the probability of large deviations in a more restrictive setting and uses completely different methods. \emr{In classical Probability Theory Large Deviation Type Estimates are also known for nonstationary sequences of independent random variables. This fact is often referred to as G\"artner-Ellis theorem (see, for example \cite{DZ}).}

For the nonstationary deterministic setting some results are also known (sometimes such systems are referred to as {\it sequential}). For instance, in \cite{BB} authors deal with the situation when all transformations have a common invariant measure, in \cite{CR} authors consider nonstationary sequences of piecewise expanding transformations of $[0, 1]$, \emr{and in \cite{GBK} authors generalize Krylov-Bogolyubov theorem and other classical results from ergodic theory of smooth dynamical systems for a non autonomous (non stationary) dynamical system, defined by a sequence of uniformly convergent maps.}

Ergodic Theorem for random \emr{nonstationary} systems has also been studied from the point of view of Markov processes (and is referred to as strong law of large numbers). For various results on that topic see \cite{G, LL1, LL2, LY, Y1, Y2, ZYW}. Large Deviation Type Estimates for \emr{inhomogeneous} Markov chains can be found in \cite{DS}. A recent monograph by Dolgopyat and Sarig \cite{DSa} investigates various limit theorems including Large Deviations Theorem for nonstationary Markov processes \emr{(see also references therein)}. Finally, an abstract version of ergodic theorem that we use as a black box to establish our result (see Theorem \ref{thm:nonstErg} below) was recently obtained by Gorodetski and Kleptsyn in \cite{GK}.

\subsection{Random walks on groups}

\emr{The literature concerning random walks on groups is vast and goes far beyond the scope of our paper, for example, one can see surveys \cite{BQ, BL, Fu} and references therein. In this section we will concentrate on random walks on compact groups only.} Limit theorems for this case have been studied intensively since 1940s. For a detailed survey of results about random walks on flat tori see \cite{DF}. Generalizing results of L\'evy \cite{L} for a random walk on $\R/\Z$ several convergence results for the sequence $\mgr^{*n}$ were obtained in pioneering works of Kawada and It\^o \cite{KI}, Kawada \cite{Ka}, Pr\'ekopa, R\'enyi and Urbanik \cite{PRU} and Urbanik \cite{U}. In the work \cite{Kl} by Kloss one will find a uniform stationary version of Ergodic Theorem and first results on the weak-* convergence of a nonstationary sequence $\mgr_{n} * \mgr_{n - 1} * \ldots * \mgr_{1}$. In \cite{S} Stromberg proves the most general result about weak-* convergence in the stationary case (see Remark \ref{rem:statWeakConv} below). In \cite{BE} it is proved (in noncommutative case) that Ergodic Theorem for a stationary random walk is equivalent to the smallest closed subgroup containing $\supp(\mgr)$ being equal to the whole group (such measures $\mgr$ are usually referred to as {\it adapted}). Results about convergence of $\mgr^{*n}$ in the sense of total variation can be found in \cite{B, RX, AG}. Finally, Borda in the recent works \cite{Bo1, Bo2} established the Central Limit Theorem and Law of Iterated Logarithm for the stationary random walks on groups. The works mentioned above also include several nonstationary forms of the ergodic theorem. However, all those theorems impose some additional Doeblin-type conditions on measures $\mgr_1, \mgr_2, \ldots, \mgr_n$, which is not required for our approach.

Our proof of Theorem \ref{thm:mainGlobal} consists of two independent parts: we prove that under suitable conditions the sequence of measures $\mgr_n * \mgr_{n - 1} * \ldots * \mgr_1$ converges (uniformly with respect to the choice of $\mgr_1, \mgr_2, \ldots, \mgr_n$) to $h$ in weak-* topology (see Theorem \ref{thm:mainRes}), and then apply a recent result by Gorodetski and Kleptsyn (see Theorem \ref{thm:nonstErg}) to deduce the ergodic theorem. Even though the weak-* convergence in Theorem \ref{thm:mainRes} is similar in spirit to the classical theorem by Kawada and It\^o (see \cite[Theorem 8]{KI}), our proof uses a completely different technique. In their work Kawada and It\^o used the machinery of representation theory to prove weak-* convergence in the stationary case. However, we don't know if these methods can be pushed further to establish a uniform nonstationary result like the one in Theorem \ref{thm:mainRes}.

\section{Proof of the main result}

Recently, A. Gorodetski and V. Kleptsyn proved a nonstationary analog of the stochastic version of Birkhoff Ergodic Theorem \cite{GK}. \emr{In order to state it we would first need to recall the definition of the Wasserstein metric. Let $X$ be a compact metric space and $\mM$ be the space of Borel probability measures on $X$. Then the Wasserstein metric can be defined as follows:}

\begin{defn} \label{def:Wass}
    Let $\msp_1, \msp_2$ be two probability measures in $\mM$. Then the \emph{Wasserstein distance} between them is defined as
    \begin{equation*}
        W(\msp_1, \msp_2) = \inf_{\gamma} \iint_{\emr{X \times X}} d(x, y) \dd \gamma(x, y),
    \end{equation*}
    where the infimum is taken over all probability measures $\gamma$ on $(\emr{X \times X}, \sigma( \mathcal{B} \times \mathcal{B} ))$ with the marginals (projections on the $x$ and $y$ coordinates) $P_x(\gamma) = \msp_1$ and $P_y(\gamma) = \msp_2$.
\end{defn}

\begin{rem}
    It is well known (see, for example \cite[Theorem 6.9]{Vi}) that convergence in Wasserstein metric is equivalent to the one in weak-* topology.
\end{rem}

\emr{
\begin{rem} \label{rem:conv}
    Another important property of Wasserstein metric that we will use is the following convexity:
    \begin{equation*}
        W(t \msp_1 + (1 - t) \msp_2, \msp_3) \le t W(\msp_1, \msp_3) + (1 - t) W (\msp_2, \msp_3) \quad \forall \msp_1, \msp_2, \msp_3 \in \mM, \; \forall t \in (0, 1).
    \end{equation*}
    This statement follows directly from Definition \ref{def:Wass}.
\end{rem}
}

Let us also introduce the following

\vspace{5pt}
{\bf Standing Assumption:} {\it \emr{Let $X$ be a compact metric space.} We will say that a sequence of distributions
$\mgr_1, \mgr_2, \mgr_3, \ldots$ on $C(X, X)$ satisfies the Standing Assumption if for any $\delta>0$ there exists $m\in \mathbb{N}$ such that the images of any two initial measures after averaging over $m$ random steps after any initial moment $n$ become $\delta$-close to each other in terms of Wasserstein distance:
\begin{equation*}
    W (\mgr_{n+m}*\ldots *\mgr_{n+1} *\msp, \mgr_{n+m}*\ldots *\mgr_{n+1} *\msp') < \delta
\end{equation*}
for any probability measures $\msp$ and $\msp'$ on $X$ and any $n \in \N$.
}
\vspace{5pt}

Now the nonstationary stochastic version of Birkhoff Ergodic Theorem can be formulated as follows :
\begin{thm}[{\cite[Theorem 3.1]{GK}}]\label{thm:nonstErg}
    \emr{Let $X$ be a compact metric space.} Suppose the sequence of distributions $\mgr_1, \mgr_2, \mgr_3, \ldots$ satisfies the \textbf{Standing Assumption} above. Given any Borel probability measure $\msp_0$ on $X$, define
    \begin{equation*}
        \msp_{n}:=\mgr_n*\msp_{n-1}, \quad n=1,2,\dots.
    \end{equation*}

    Then for any $\varphi\in C(X, \mathbb{R})$ and any $x\in X$, almost surely
    \begin{equation} \label{nonstErg}
        \frac{1}{n}\left|\sum_{k = 0}^{n - 1} \varphi(f_k\circ \ldots \circ f_1(x))-\sum_{k = 0}^{n - 1} \int_X \varphi \dd\msp_k \right|\to 0 \ \ \text{\rm as}\ \ \ n\to \infty.
    \end{equation}

    Moreover, an analogue of the Large Deviations Theorem holds. Namely, for any $\eps > 0$ there exist $C, \delta > 0$ such that for any $x \in X$ and any $n \ge 1$
    \begin{equation} \label{nonstLD}
        \p \left( \frac{1}{n}\left|\sum_{k = 0}^{n - 1} \varphi(f_k\circ \ldots \circ f_1(x))-\sum_{k = 0}^{n - 1} \int_X \varphi \dd\nu_k \right| > \eps \right) < C \exp(-\delta n),
    \end{equation}
    where
    \begin{equation*}
        \p = \prod_{n=1}^\infty \mgr_n.
    \end{equation*}
\end{thm}

\emr{Choosing $X = \A$ we see that to prove Theorem \ref{thm:mainGlobal} it is enough to verify the \textbf{Standing Assumption} for the sequence $\mgr_1, \mgr_2, \mgr_3, \ldots$, which brings us to the next}

\begin{thm} \label{thm:mainRes}
    Consider a compact (in weak-* topology) $K \subset \mM$, such that any $\mgr \in K$ is \textbf{strictly aperiodic}. Then for any $\eps > 0$ there exists $m \in \N$, such that for any sequence of measures $\mgr_1, \mgr_2, \ldots, \mgr_m \in K$ and any starting measure $\msp \in \mM$ the following inequality holds:
    \begin{equation} \label{mainIneq}
        W(\mgr_{m} * \mgr_{m - 1} * \ldots * \mgr_{1} * \msp, h) < \eps.
    \end{equation}
\end{thm}

\begin{rem} \label{rem:statWeakConv}
    A classical theorem by Stromberg \cite{S} states that \emr{in a stationary case} a weak-* limit $\lim_{n \to \infty} \mgr^{*n}$ exists if and only if $\mgr$ is \textbf{strictly aperiodic}. Moreover, said limit is always equal to the Haar measure $h$. It can be easily seen from Stromberg's result that Theorem \ref{thm:mainRes} will not hold if there is a not \textbf{strictly aperiodic} $\mgr \in K$. In fact, a simple counterexample can be constructed for $\A = \R / \Z$: consider $\mgr = \delta_{\alpha}$ for some irrational $\alpha$. Despite the fact that $\bigcup_{n \in \N} \supp(\mgr^{*n})$ is dense in $\R / \Z$, there is still no weak-* limit for $\mgr^{*n}$.

    The compactness of $K$ is also essential. For the same $\A = \R / \Z$ fix some irrational $\alpha \in (0, 1)$ and define
    \begin{equation*}
        \mgr_{n} = \frac{1}{2} \delta_{0} + \frac{1}{2} \delta_{\alpha/2^n}.
    \end{equation*}
    Even though every $\mgr_n$ is \textbf{strictly aperiodic} we have
    \begin{equation*}
        \supp (\mgr_n * \mgr_{n - 1} * \ldots * \mgr_{1}) \subset [0, \alpha],
    \end{equation*}
    hence the weak-* limit of $\mgr_n * \mgr_{n - 1} * \ldots * \mgr_{1}$ cannot be equal to $h = \Leb|_{[0, 1)}$.
\end{rem}

Postponing the proof of Theorem \ref{thm:mainRes} until the next section we would like to deduce Theorem \ref{thm:mainGlobal}:

\begin{proof} [Proof of Theorem \ref{thm:mainGlobal}]
    It follows from inequality (\ref{mainIneq}) that for any sequence \\
    $\mgr_1, \mgr_2, \mgr_3 \ldots \in K$ the \textbf{Standing Assumption} holds. Applying Theorem \ref{thm:nonstErg} we obtain that statements (\ref{nonstErg}) and (\ref{nonstLD}) hold. Since convergence in Wasserstein metric is equivalent to weak-* convergence inequality (\ref{mainIneq}) also implies that
    \begin{equation*}
        \frac{1}{n} \sum_{k = 0}^{n - 1} \int_{\A} \varphi \dd \msp_k \to \int_{\A} \varphi \dd h.
    \end{equation*}
    After that formulae (\ref{isomErg}) and (\ref{isomLD}) follow directly from (\ref{nonstErg}) and (\ref{nonstLD}).
\end{proof}

\section{Proof of Theorem \ref{thm:mainRes}: weak-* convergence to the Haar measure}

We will denote by $B(x, r)$ an open ball in metric $d$ of radius $r$ centered at $x$. We would like to start with a following
\begin{defn} \label{def:epsWide}
    We say that a set $Q \in \mathcal{B}$ is $\eps$-wide if it is contained in a ball of radius $\eps$ and contains a ball of radius $\eps/3$.
\end{defn}

\begin{rem}
    Note that it follows directly from Vitali covering lemma \cite{Vit} that any compact metric space can be partitioned into finitely many $\eps$-wide sets for any fixed $\eps > 0$.
\end{rem}

Let us postpone the proof of the following technical statement until Section \ref{sec:Tec}:

\begin{prop} \label{prop:right_form}
    \emr{Let $K$ be a compact set in $\mM$, such that any $\mgr \in K$ is {\bf strictly aperiodic}.} For any $\eps > 0$ there exists $m \in \N$ and $\delta > 0$, such that for any $\mgr_1, \mgr_2, \ldots, \mgr_m \in K$, any $\msp \in \mM$ and any $\eps$-wide set $Q$ we have
    \begin{equation} \label{mixingCond}
        [\mgr_{m} * \mgr_{m - 1} * \ldots * \mgr_{1} * \msp] (Q) \ge \delta.
    \end{equation}
\end{prop}

Before proving Theorem \ref{thm:mainRes} let us establish the following
\begin{lm} \label{lm:mon}
    For any $\mgr, \msp \in \mM$ one has
    \begin{equation} \label{shift_mon}
        W(\mgr * \msp, h) \le W(\msp, h).
    \end{equation}
\end{lm}

\begin{proof}[Proof of Lemma \ref{lm:mon}]
    Let $\gamma$ be a probability measure on $\A \times \A$, such that $P_x(\gamma) = \msp$ and $P_y(\gamma) = h$. Define a measure
    \emr{\begin{equation*}
        \dd \tilde{\gamma} (x, y) = \int_{\A} \dd\gamma(x - \alpha, y - \alpha) \dd\mgr(\alpha).
    \end{equation*}
    It is clear that $P_x(\tilde{\gamma}) = \mgr * \msp$ and $P_y(\tilde{\gamma}) = h$. Now
    \begin{multline*}
        \iint_{\A \times \A} d(x, y) \dd\tilde{\gamma} (x, y) = \\
        = \iiint_{\A \times \A \times \A} d(x, y) \dd\gamma(x - \alpha, y - \alpha) \dd\mgr(\alpha) =\\
        = \iiint_{\A \times \A \times \A} d(x + \alpha, y + \alpha) \dd\gamma(x, y) \dd\mgr(\alpha) =\\
        = \iiint_{\A \times \A \times \A} d(x, y) \dd\gamma(x, y) \dd\mgr(\alpha) = \\
        = \iint_{\A \times \A} d(x, y) \dd\gamma (x, y)
    \end{multline*}}
    and inequality (\ref{shift_mon}) follows.
\end{proof}

\begin{proof}[Proof of Theorem \ref{thm:mainRes}]

    Let us fix $\eps > 0$ and $\msp \in \mM$. We will first show that for $m \in \N$ and $\delta$ satisfying (\ref{mixingCond}) the measure $\msp^{(m)} := \mgr_{m} * \mgr_{m - 1} * \ldots * \mgr_{1} * \msp$ can be decomposed into a sum
    \begin{equation} \label{meas_decomp}
        \msp^{(m)} %= \mgr_{m} * \mgr_{m - 1} * \ldots * \mgr_{1} * \msp
        = \msp_0 + \msp_1,
    \end{equation}
    where $\msp_1(\A) = \delta$ and
    \begin{equation} \label{msp2dist}
        W\left( \frac{\msp_1}{\msp_1(\A)}, h \right) \le \eps.
    \end{equation}
    In order to construct that decomposition we fix a partition of $\A$ into $\eps$-wide sets $\{Q_j\}_{1 \le j \le k}$. Next, define a measure
    \begin{equation*}
        \dd \nu_1 (x) = \sum_{j = 1}^{k} \frac{\delta h(Q_j)}{\msp^{(m)}(Q_j)} \1_{Q_j} (x) \dd \msp^{(m)} (x).
    \end{equation*}
    Notice that according to (\ref{mixingCond}) we have $\frac{\delta h(Q_j)}{\msp^{(m)}(Q_j)} < \frac{\delta}{\msp^{(m)} (Q_j)} \le 1$, so $\msp_0 = \msp^{(m)} - \msp_1$ is a well-defined positive measure and $\msp_1(Q_j) = \delta h(Q_j)$. Clearly, $\msp_1(\A) = \delta$. Now let us show that (\ref{msp2dist}) holds. We would like to construct a probability measure $\gamma$ on $\A \times \A$, such that
    \begin{equation} \label{gammaSupp}
        \supp(\gamma)~\subset~\cup_{j = 1}^{k} Q_j \times Q_j,
    \end{equation}
    $P_x(\gamma) = \frac{\msp_1}{\delta}$, and $P_y(\gamma) = h$. It immediately follows from (\ref{gammaSupp}) that
    \begin{multline*}
        \iint_{\A \times \A} d(x, y) \dd\gamma(x, y) = \sum_{j = 1}^{k} \iint_{Q_j \times Q_j} d(x, y) \dd\gamma(x, y) \le \\
        \le \eps \sum_{j = 1}^{k} \gamma(Q_j \times Q_j) = \eps.
    \end{multline*}
    The construction of required $\gamma$ can be done in the following way: for each $Q_j$ define
    \begin{equation*}
        \gamma|_{Q_j \times Q_j} = \dfrac{\frac{\msp_1|_{Q_j}}{\delta} \otimes h|_{Q_j}}{h(Q_j)}.% = \frac{\frac{\msp_2|_{Q_j}}{p} \otimes h|_{Q_j}}{\msp_2(Q_j)}.
    \end{equation*}
    It's easy to verify that $\gamma$ satisfies required properties and hence (\ref{msp2dist}) holds.

    \emr{Now we can apply the same reasoning to the measure $\msp_0$. Namely, after applying another $m$ iterations of our random walk we will be able to decompose $\mgr_{2m} * \mgr_{2m - 1} * \ldots * \mgr_{m + 1} * \msp_0$ the same way as in (\ref{meas_decomp}) (up to scaling). That will give the following decomposition for $\msp^{(2m)} := \mgr_{2m} * \mgr_{2m - 1} * \ldots * \mgr_{1} * \msp $:
    \begin{equation*}
        \msp^{(2m)} = \msp^{(2)}_0 + \msp^{(2)}_1 + \msp^{(2)}_2,
    \end{equation*}
    where $\msp^{(2)}_{1} = \mgr_{2m} * \mgr_{2m - 1} * \ldots * \mgr_{m + 1} * \msp_1$, $\msp^{(2)}_2(\A) = \delta (1 - \delta)$ and 
    \begin{equation*}
        W \left( \frac{\msp^{(2)}_2}{\msp^{(2)}_2(\A)}, h \right) \le \eps.
    \end{equation*}
    In addition, combining Lemma \ref{lm:mon} and inequality (\ref{msp2dist}) we get 
    \begin{equation*}
        W \left( \frac{\msp^{(2)}_1}{\msp^{(2)}_1(\A)}, h \right) \le \eps.
    \end{equation*}
    Repeating this procedure $r > \log_{1 - \delta} (\eps)$ times we will obtain a decomposition
    \begin{equation*}
        \msp^{(mr)} = \msp^{(r)}_0 + \msp^{(r)}_1 + \ldots + \msp^{(r)}_{r},
    \end{equation*}}
    such that
    \begin{equation} \label{mspjdist}
        W\left( \frac{\msp^{(r)}_j}{\msp^{(r)}_j(\A)}, h \right) = W\left( \frac{\msp^{(r)}_j}{(1 - \delta)^{j - 1} \delta}, h \right) \le \eps \quad \text{for any $1 \le j \le r$}
    \end{equation}
    and $\msp^{(r)}_0(\A) = (1 - \delta)^r < \eps$. Finally, \emr{applying convexity of $W$ stated in Remark \ref{rem:conv}}, we observe that
    \begin{multline*}
        W \left( \sum_{j = 0}^{r} \msp^{(r)}_j, h \right) \le \msp^{(r)}_0(\A) W \left( \frac{\msp^{(r)}_0}{\msp^{(r)}_0(\A)}, h \right) + \sum_{j = 1}^{r} \msp^{(r)}_j(\A) W \left( \frac{\msp^{(r)}_j}{\msp^{(r)}_j(\A)}, h \right) \le \\
        \le \Diam(\A) \eps + \eps = (\Diam(\A) + 1) \eps
    \end{multline*}
    and since $\eps$ was chosen an arbitrary positive constant Theorem \ref{thm:mainRes} is proven.

\end{proof}

\section{Proof of Proposition \ref{prop:right_form}: $\eps$-density of the support} \label{sec:Tec}

We would like to start by proving the following
\begin{lm} \label{lm:OneMesDenseSupp}
    Consider $\mgr \in \mM$ that satisfies the \textbf{strict aperiodicity} condition. Then for any $\eps > 0$ there exists $m \in \N$, such that $\supp(\mgr^{*m})$ is $\eps$-dense in $\A$.
\end{lm}

\begin{proof}

    It follows from \textbf{strict aperiodicity} of $\mgr$ that for a fixed $\eps > 0$ there exists $\emr{m_0} \in \N$, such that the set of points of the form $\sum_{j = 1}^n \alpha_j - \beta_j$ with $\alpha_j, \beta_j \in \supp(\mgr)$ and $n \le \emr{m_0}$ is $\eps$-dense in $\A$. Let us choose a finite number of these sums $\sum_{j = 1}^{n_k} \alpha_j^{(k)} - \beta_j^{(k)}$ for $1 \le k \le N$ that forms a finite $\eps$-net in $\A$. First of all, notice that we can make $n_k =\emr{ m_0}$ for all $1 \le k \le N$ by adding some bogus terms of the form $\alpha - \alpha$ for some $\alpha \in \supp(\mgr)$. Define an angle
    \begin{equation*}
        \beta = \sum_{k = 1}^{N} \sum_{j = 1}^{\emr{m_0}} \beta_{j}^{(k)}.
    \end{equation*}
    Then all the sums of the form $\sum_{j = 1}^{\emr{m_0}} (\alpha_j^{(k)} - \beta_j^{(k)}) + \beta$ are elements of $\supp(\mgr^{*\emr{m_0}N})$. Hence $\supp(\mgr^{*\emr{m_0}N}) + \beta$ is $\eps$-dense in $\A$ and so is $\supp(\mgr^{*\emr{m_0}N})$, \emr{hence $m = m_0 N$ satisfies the required statement}.

\end{proof}

\begin{rem} \label{rem:dens_pres}
    Notice that if $\supp(\mgr)$ is $\eps$-dense then so is $\supp(\tilde{\mgr} * \mgr) = \supp(\mgr * \tilde{\mgr})$ for any $\mgr, \tilde{\mgr} \in \mM$.
\end{rem}

Now, consider a compact set $K \subset \mM$ such that any measure $\mgr \in K$ has the \textbf{strict aperiodicity} property and let us prove the key property of such compact set:
\begin{prop} \label{prop:measure_erg}
    For any $\eps > 0$ there exists $m$, such that for any sequence\\ ${\mgr_1, \mgr_2, \ldots, \mgr_m \in K}$ the support $\supp(\mgr_m * \mgr_{m - 1} * \ldots * \mgr_{1})$ is $\eps$-dense in $\A$.
\end{prop}

%\newpage
We will start by proving the following
\begin{lm} \label{lm:delta_erg}
    For any $\eps > 0$ and any measure $\mgr \in K$ there exist $\delta > 0$ and $m \in \N$, such that for any sequence of measures $\mgr_1, \mgr_2, \ldots, \mgr_m \in K$, such that $W(\mgr, \mgr_j) < \delta$ the support $\supp(\mgr_m * \mgr_{m - 1} * \ldots * \mgr_{1})$ is $\eps$-dense in $\A$.
\end{lm}

\begin{proof}[Proof of Lemma \ref{lm:delta_erg}]
    Since $\mgr \in K$ it is \textbf{strictly aperiodic}, we can apply Lemma \ref{lm:OneMesDenseSupp} and conclude that for a fixed $\eps$ there exists $m$, such that $\supp(\mgr^{* m})$ is $\eps/2$-dense in $\A$. We would like to denote $\tilde{\eps} = \eps / (2 m)$ and prove the following technical
    \begin{lm} \label{lm:clSupp}
        There exists $\delta > 0$, such that for any $\tilde{\mgr}$, such that $W(\mgr, \tilde{\mgr}) < \delta$ we have $\supp(\mgr) \subset \supp(\tilde{\mgr}) + B(0, \tilde{\eps})$.
    \end{lm}
    \begin{proof}[Proof of Lemma \ref{lm:clSupp}]
        Assume such $\delta$ does not exist. Then for any positive $\delta > 0$ there is a measure $\tilde{\mgr}_{\delta}$ and a point $x_{\delta}$, such that $x_{\delta} \in \supp(\mgr)$ but $x_{\delta} \notin \supp(\tilde{\mgr}_{\delta}) + B(0, \tilde{\eps})$. \emr{Then $B(x_{\delta}, \tilde{\eps}) \cap \supp(\tilde{\mgr}_{\delta}) = \varnothing$ and hence $\mgr(B(x_{\delta}, \tilde{\eps}/2)) \cdot \tilde{\eps} / 2 \le W(\mgr, \tilde{\mgr})$. The last inequality is true because for any measure $\gamma$ on $\A \times \A$, such that $P_x(\gamma) =\mgr$ and $P_y(\gamma) = \tilde{\mgr}_\delta$ we will have
        \begin{multline*}
            \iint_{\A \times \A} d(x, y) \dd \gamma(x, y) \ge \iint_{B(x_{\delta}, \tilde{\eps}/2) \times (\A \setminus B(x_{\delta}, \tilde{\eps}))} d(x, y) \dd \gamma(x, y) \ge \\
            \ge \frac{\tilde{\eps}}{2} \cdot \mgr({B(x_{\delta}, \tilde{\eps}/2)}) \cdot \tilde{\mgr}_{\delta} (\A \setminus B(x_{\delta}, \tilde{\eps}/2)) = \frac{\tilde{\eps}}{2} \cdot \mgr({B(x_{\delta}, \tilde{\eps}/2)}).
        \end{multline*}}
        Recalling that $W(\mgr, \tilde{\mgr}_{\delta}) < \delta$ we arrive to
        \begin{equation*}
            \mgr(B(x_{\delta}, \tilde{\eps}/2)) < \frac{2 \delta}{\tilde{\eps}}.
        \end{equation*}
        Taking $\delta \to 0$ and choosing a convergent subsequence of $x_{\delta}$ by compactness we find a point ${x_0 \in \supp(\mgr)}$, such that $\mgr(B(x_{0}, \tilde{\eps}/2)) = 0$, which contradicts the definition of $\supp({\mgr})$.
    \end{proof}
    Applying Lemma \ref{lm:clSupp} to every individual measure $\mgr_{j}$ for $1 \le j \le m$ we see that for the same $\delta$ we have
    \begin{equation*}
        \supp(\mgr^{* m}) \subset \supp(\mgr_m * \mgr_{m - 1} * \ldots * \mgr_{1}) + B(0, \eps/2),
    \end{equation*}
    which fact guarantees that $\supp(\mgr_m * \mgr_{m - 1} * \ldots * \mgr_{1})$ is $\eps$-dense in $\A$.
\end{proof}

\begin{proof}[Proof of Proposition \ref{prop:measure_erg}]
    Let us fix some $\eps > 0$. First, since $K$ is compact we can choose such $m_0 \in \N$ and $\delta > 0$ that for any sequence of measures $\mgr_1, \mgr_2, \ldots, \mgr_{m_0} \in K$ with $W(\mgr_j, \mgr_k) < \delta$ the support $\supp(\mgr_{m_0} * \mgr_{m_0 - 1} * \ldots * \mgr_{1})$ will be $\eps$-dense in $\A$. We can also choose a finite covering of $K$ by balls of radius $\delta/2$. Assuming that covering has $n$ balls let us take $m = n m_0$. Then any sequence of $m$ measures will have at least $m_0$ terms $\mgr_{n_1}, \mgr_{n_2}, \ldots, \mgr_{n_{m_{0}}}$ from the same ball of radius $\delta/2$ and hence $W(\mgr_{n_j}, \mgr_{n_k}) < \delta$ will hold. Due to commutativity we can rearrange the measures in a following way:
    \begin{equation*}
        \mgr_{m} * \mgr_{m - 1} * \ldots * \mgr_1 = \ldots * \mgr_{n_{m_0}} * \mgr_{n_{m_0 - 1}} * \ldots * \mgr_{n_1}.
    \end{equation*}
    The support $\supp(\mgr_{n_{m_0}} * \mgr_{n_{m_0 - 1}} * \ldots * \mgr_{n_1})$ is $\eps$-dense because of our choice of $\delta$. Applying Remark \ref{rem:dens_pres} we conclude that the support $\supp(\mgr_{m} * \mgr_{m - 1} * \ldots * \mgr_1)$ is also $\eps$-dense.
\end{proof}

\emr{Using Proposition \ref{prop:measure_erg} we can finally finish the proof of Proposition \ref{prop:right_form}. Assume that it is not true and fix $\eps > 0$ for which the statement fails. Then by Proposition \ref{prop:measure_erg} there exists $m \in \N$, such that for any sequence $\mgr_1, \mgr_2, \ldots, \mgr_m \in K$ the support $\supp(\mgr_m * \mgr_{m - 1} * \ldots * \mgr_{1})$ is $\eps/8$-dense in $\A$. Let us fix that $m$. By our assumption for any $n \in \N$ there exists a sequence $\mgr_{1}^{(n)}, \mgr_{2}^{(n)}, \ldots, \mgr_{m}^{(n)} \in K$, $\msp^{(n)} \in \mM$ and an $\eps$-wide set $Q_n$, such that
\begin{equation*}
    [\mgr_{m}^{(n)} * \mgr_{m - 1}^{(n)} * \ldots * \mgr_{1}^{(n)} * \msp^{(n)} ] (Q_n) < \frac{1}{n}.
\end{equation*}
Due to compactness of $\A$ we can choose an infinite subsequence $Q_{n_k}$, such that all $Q_{n_k}$ cover the same ball $B$ of radius $\eps/4$ (see Definition \ref{def:epsWide}). After that we can choose a convergent subsequence $m + 1$ more times, so that $\mgr_{j}^{(n)} \to_{n \to \infty} \mgr_j$ for $1 \le j \le m$  and $\msp^{(n)} \to_{n \to \infty} \msp$. Now we see that for some $\mgr_1, \mgr_2, \ldots, \mgr_m \in K$ and $\msp \in \mM$ we have
\begin{equation*}
    [\mgr_{m} * \mgr_{m - 1} * \ldots * \mgr_{1} * \msp] (B) = 0,
\end{equation*}
which contradicts the fact that $\supp(\mgr_{m} * \mgr_{m - 1} * \ldots * \mgr_{1})$ (and hence, by Remark \ref{rem:dens_pres} also $\supp(\mgr_{m} * \mgr_{m - 1} * \ldots * \mgr_{1} * \msp)$) is $\eps/8$-dense. That concludes the proof of Proposition \ref{prop:right_form}.}

\section*{Acknowledgements}

The author is grateful to A. Gorodetski and V. Kleptsyn for permanent attention to the work and fruitful discussions, to S. Klein and D. Dolgopyat for providing helpful references, and both anonymous reviewers for their detailed feedback and valuable comments. The author was supported in part by NSF grant DMS--2247966 (PI: A.\,Gorodetski).

\bibliographystyle{amsplain}

\end{document}